\documentclass{amsart}
\usepackage[latin1]{inputenc}
\usepackage{amsmath}
\usepackage{enumerate}
\usepackage{mathrsfs}
\usepackage{amsfonts}
\usepackage{amssymb}
\usepackage{graphicx}
\usepackage{amsthm}
\usepackage[toc,page]{appendix}
\usepackage{geometry}
\usepackage[all]{xy}
\usepackage{lmodern}
\usepackage[T1]{fontenc}
\usepackage[dvipsnames]{xcolor}
\usepackage{graphics}
\usepackage{comment}
\usepackage{csquotes}
\usepackage[colorlinks=true,linkcolor=NavyBlue,urlcolor=RoyalBlue,citecolor=PineGreen,linktocpage=true]{hyperref}
\usepackage{graphicx}

\numberwithin{equation}{section}

\newtheorem{Theorem}{Theorem}[section]	
\newtheorem*{Theorem*}{Theorem}
\newtheorem{proposition}[Theorem]{Proposition} 
\newtheorem{lemma}[Theorem]{Lemma}

\newtheorem*{defn*}{Definition}
\newtheorem{Cor}[Theorem]{Corollary}
\newtheorem*{remark}{Remark}

\newtheorem{example}{Example}[section]

\usepackage{tikz}
\usetikzlibrary{decorations,arrows}
\usetikzlibrary{decorations.pathmorphing} 
\usepackage{pgf}
\usetikzlibrary{patterns}

\tikzset{
	schraffiert/.style={pattern=horizontal lines,pattern color=#1},
	schraffiert/.default=black
}

\tikzset{
	ultra thin/.style= {line width=0.1pt},
	very thin/.style=  {line width=0.2pt},
	thin/.style=       {line width=0.4pt},% thin is the default
	semithick/.style=  {line width=0.6pt},
	thick/.style=      {line width=0.8pt},
	very thick/.style= {line width=1.2pt},
	ultra thick/.style={line width=2.4pt}
}

\usetikzlibrary{shapes}

\title{On measure estimates arising from Hausdorff distance convergence}
\author{Lior Tenenbaum}
\date{}

\address{Department of Mathematics\\
	Bar-Ilan university\\
	Ramat Gan, Israel}
\email{lior.tenenbaum@biu.ac.il}

\begin{document}
	\begin{abstract}
		We discuss a method to estimate the measure of a compact set which is approximated using the Hausdorff distance by a sequence of compact sets. We do this by considering corresponding fattenings of the sequence of compact sets and showing their measures converge. We further review applications of this result to study the measure of a spectrum of an operator which has a sequence of periodic approximations. 
	\end{abstract}
	
	\maketitle
\section{Introduction} \label{Sect:Intro}

Given a metric space $(X,d)$, a natural way to quantify the convergence of subsets in $X$ is via the Hausdorff distance,  see \cite[Chapter 7]{BBI01}. 
The \emph{Hausdorff distance} between two nonempty subsets $A,B\subseteq X$ is given by
\begin{equation} \label{Eq:HausDistDefn}
	d_H(A,B):= \max \big\{ \sup_{a\in A} d(a,B),   \sup_{b\in B} d(b,A) \big\},
\end{equation}
where $d(a,B)= \inf_{b\in B} d(a,b)$. 
The restriction of $d_H$ to $\mathcal{K}(X)$, the collection of nonempty compact subsets of $X$, is a metric.
Given a Borel measure $\mu$ on $X$, a natural question is when does  convergence with respect to the Hausdorff distance implies convergence of the associated measures. More precisely, given a Borel measure $\mu$ and a sequence of compact subsets $\{ A_n \}_{n=1}^\infty\subseteq \mathcal{K}(X)$ satisfying $\lim_{n\to \infty}d_H( A_n,A_\infty )=0$ for some $A_\infty\in \mathcal{K}(X)$, does it follow that $\lim_{n\to \infty} \mu(A_n)=\mu(A_\infty)$? \\
In general, this implication fails, as the following example shows.

\begin{example} \label{Examp:Simple}
	Consider the Lebesgue measure $\operatorname*{Leb}$ on $\mathbb{R}$ with the standard norm distance. Let $A_\infty:=[0,1]$ and 
	\[ A_n:=\Big\{ \frac{j}{n}: j\in \mathbb{Z}  \Big\} \cap [0,1]. \] 
	Then $\lim_{n\to \infty}d_H(A_n,A_\infty)=0$, while $\operatorname*{Leb}(A_\infty)=1$ and $\operatorname*{Leb}(A_n)=0$ for all $n\in \mathbb{N}$. Similarly, for each $0\leq\alpha<1$, define a sequence $A_n(\alpha):=A_n\cup[0,\alpha]$ for all $n\in \mathbb{N}$. Then $\lim_{n\to \infty} \operatorname*{Leb}\big( A_n(\alpha) \big)=\alpha\neq \operatorname*{Leb}(A_\infty)$, while $\lim_{n\to \infty} d_H\big( A_n(\alpha), A_\infty  \big)=0$.
\end{example}

	In fact, 
	for any compact metric space $(X,d)$ and any positive, non-atomic Borel measure $\mu$ on $X$,
	one can construct analogous counterexamples to those in Example \ref{Examp:Simple}. For the Lebesgue measure on  $\mathbb{R}^d$, a characterization of when Hausdorff convergence implies measure convergence was obtained in \cite{Beer74}. 
	In this note, we establish measure estimates for compact sets that arise as Hausdorff limits, even in cases where the corresponding measures fail to converge. % using set fattenings. 
	\\
	
	The Hausdorff distance can also be expressed in terms of \emph{fattenings} or \emph{thickenings}.
	Given a metric space $(X,d)$ and a nonempty subset $A\subseteq X$, define for each $\delta\geq 0$ the $\delta$-fattening of $A$ by
	\begin{equation} \label{Eq:SetFattDefn}
		A^{(\delta)}:= \bigcup_{a\in A} \overline{ B }_\delta(a),
	\end{equation}
	where $\overline{ B }_\delta(a):=\{ x\in X : d(x,a)\leq \delta  \}$. It is well-known that, for nonempty subsets $A,B\subseteq X$, 
	\begin{equation} \label{Eq:SetFattEquiv}
		d_H(A,B) = \inf\{ \eta>0 : B\subseteq A^{(\eta)} \; \text{and} \;  A\subseteq B^{(\eta)} \}.
	\end{equation}
	Moreover, if $\delta=d_H(A,B)$, then $A\subseteq B^{(\delta)}$ and $B\subseteq A^{(\delta)}$,  provided $\delta>0$ or both $A,B\subseteq X$ are closed. 
	Returning to Example \ref{Examp:Simple}, one computes explicitly that $\delta_n:=d_H(A_n,A_\infty)=\frac{1}{2n}$ and
	\[ A_\infty^{(\delta_n)}=\Big[ -\frac{1}{2n}, 1+\frac{1}{2n} \Big]=A_n^{(\delta_n)}.  \]
	Hence, 
	\begin{equation*}
		\operatorname*{Leb}( A_n^{(\delta_n)} )= 1+\frac{2}{2n} \overset{n\to \infty}{\longrightarrow} 1= \operatorname*{Leb}(A_\infty).
	\end{equation*}
	This observation extends to a much more general setting, as described in the following theorem. Before stating the theorem, we introduce additional needed notions. A metric space $(X,d)$ is called \emph{proper} if every closed and bounded subset of $X$ is also compact. A Borel measure $\mu$ on $X$ is said to be \emph{locally finite} if every point $x\in X$ has a neighborhood $U_x$ with $\mu(U_x)<\infty$. Since the discussed measures are locally finite measures and subadditive, every compact subset of $X$ has finite measure with respect to such $\mu$.
	\begin{Theorem} \label{Thm:Main1}
		Let $(X,d)$ be a proper metric space, and let $\mu$ be a locally finite Borel measure on $X$. Suppose that $\{A_n\}_{n=1}^\infty\subseteq \mathcal{K}(X)$ is a sequence of compact sets converging to $A_\infty \in \mathcal{K}(X)$ in the Hausdorff distance. Then one has
		\[  \lim_{n\to \infty} \mu\big( A_n^{(\delta_n)} \big) =\mu(A_\infty), \quad \text{where} \quad \delta_n:= d_H(A_n, A_\infty). \]
		%where $\delta_n:= d_H(A_n, A_\infty)$.
	\end{Theorem}
	Although the proof of this statement is elementary, it appears not to have been recorded previously, to the best of the author's knowledge. 
	%In fact, this note was written following a failure to find such a result in the mathematical literature.
	
	Our interest in this result arises from the goal of studying the spectrum of a bounded operator through approximations measured in the Hausdorff distance. 
	Specifically, given a bounded Schr\"{o}dinger operator $H:\ell^2(\mathbb{Z}^d)\to \ell^2(\mathbb{Z}^d)$, we wish to approximate its spectrum $\Sigma_\infty:= \operatorname*{spec}(H)$ by a sequence of finite or periodic Schr\"{o}dinger operators $H_n:\ell^2( \mathbb{Z}^d ) \to \ell^2(\mathbb{Z}^d)$ satisfying
	\begin{equation} \label{Eq:BandRep}
		\lim_{n\to \infty} d_H( \Sigma_n,\Sigma_\infty )=0, \quad \text{where} \quad \Sigma_n:=\text{spec}(H_n).
	\end{equation}
	This approximation scheme appears, both implicitly and explicitly, in a variety of works on aperiodic Schrödinger operators, including the Almost-Mathieu operator \cite{AvSim83,AvMoSi91,Last93,Last94}, Sturmian Hamiltonians \cite{Casdagli1986,Sut87,BIST89,Ray95,Dam08,Dam13}, and various others.
	For these operators, Bloch--Floquet theory  \cite{SimonReedIV, Korot14, Kuc16, Fil21,LiuFil22} guarantees that such a sequence of periodic Schr\"{o}dinger operators $H_n$, referred in some places as having twisted boundary conditions, satisfy 
	\begin{equation*} \label{Eq:FloquetRep}
		\operatorname*{spec}(H_n)= \bigcup_{i=1}^{q_n} [a_i^{(n)},b_i^{(n)}] \quad \text{for each} \quad n\in \mathbb{N},
	\end{equation*}  
	for some $q_n\in \mathbb{N}$ and real numbers $\{ a^{(n)}_i \}_{i=1}^{q_n},\{ b_i^{(n)} \}_{i=1}^{q_n}\subseteq \mathbb{R}$. 
	Moreover, there exists a family of self-adjoint $q_n\times q_n$ matrices $\{H_n(\theta)\}_{\theta \in [0,1)^d}$, such that the interval $[a_i^{(n)},b_i^{(n)}]$ equals the range of the $i$-th eigenvalue of $H_n(\theta)$. Consult \cite[Section 2]{Fil21} and the references therein for more details. 
	\begin{Cor} \label{Cor:Last1}
		Let $H:\ell^2(\mathbb{Z}^d)\to \ell^2(\mathbb{Z}^d)$ be a bounded self-adjoint Schr\"{o}dinger operator, and let $\mu$ be a locally finite Borel measure on $\mathbb{R}$. Assume that there exists a sequence of periodic bounded self-adjoint  Schr\"{o}dinger operators $H_n:\ell^2(\mathbb{Z}^d)\to \ell^2(\mathbb{Z}^d)$ such that 
		\[ \delta_n:= d_H\big( \operatorname*{spec}(H_n) , \operatorname*{spec}(H) \big) \to 0 \quad \text{as} \; n\to \infty.\]
		Then
		\[  \mu \big( \operatorname*{spec}(H) \big)= \lim_{n\to \infty} \mu \Big(  \bigcup_{i=1}^{q_n}  \big[ a_i^{(n)}-\delta_n, b_i^{(n)}+\delta_n \big] \Big).  \]
	\end{Cor}   
	This corollary provides a practical mean of estimating the measure of a spectrum that is approximated by periodic Schr\"{o}dinger operators. 
	Using terminology from scientific computing, the approximations in Corollary \ref{Cor:Last1} admit algorithms of $\Pi_1$ type for approximating the measure of the spectrum. See \cite{CEF24} for more details. This approach could lead to novel methods for determining spectral properties of higher dimensional discrete Schrödinger operators, see \cite[Open problem 30]{DaFi24-book_2}.
	We also note the extensive body of work deriving estimates on the measure, size, and location of periodic spectra, including \cite{Last93, GolKust19, Korot22, Fil25}.\\
	Further results concerning spectral approximations as in Corollary \ref{Cor:Last1}, including the proof of Corollary~\ref{Cor:Last1}, are presented in Section~\ref{Sect:Periodic}.
	The next section is devoted to the proof of Theorem~\ref{Thm:Main1}.
	
	\subsection*{Acknowledgments}
	The author gratefully acknowledges Ram Band, Philipp Bartmann, Siegfried Beckus, and Ron Rosenthal for many insightful discussions, and Philipp Bartmann, Jake Fillman, and Yannik Thomas for helpful comments and corrections. Special thanks are due to Siegfried Beckus for suggesting the spectral covers in Section~\ref{Sect:Periodic}, and to Jake Fillman for pointing out several related works connected to this note.\\
	The author also gratefully acknowledges the financial support of the Deutsche Forschungsgemeinschaft (DFG Grant No.~6789/1-1) and the Israel Science Foundation (ISF Grant No.~844/19).

\section{Proof of the main result} \label{Sect:ProofMain}
We build towards the proof of Theorem \ref{Thm:Main1}. For a space $X$ and a subset $A\subseteq X$, we denote the indicator function of $A$ by $\mathsf{1}_A:X\to \{0,1\}$. We use the fact that for any measure space $(X, \mathcal{B},\mu)$, we have $\mu(A)= \int_X \mathsf{1}_Ad\mu$ for every measurable set $A\in \mathcal{B}$. Using this characterization, it suffices to show
\begin{equation*}
	\lim_{n\to \infty} \int_{X} \mathsf{1}_{A_n^{(\delta_n)}}d\mu_n = \int_X \mathsf{1}_{A_\infty}d\mu 
\end{equation*}
for any sequence of compact sets $\{ A_n \}_{n=1}^\infty \subseteq \mathcal{K}(X)$ and $A_\infty\in \mathcal{K}(X)$ such that $\delta_n:=d_H( A_n, A_\infty )\to 0$ as $n\to \infty$. 
By the dominated convergence theorem (DCT), it is therefore enough to show that $\lim_{n\to \infty} \mathsf{1}_{A_n^{(\delta_n)}}(x)=\mathsf{1}_{A_\infty}(x)$ for all $x\in X$, and that there exists a dominating $\mu$-integrable function. The latter statement follows by taking $g(x)= \mathsf{1}_{A_\infty^{(2\hat{\delta})}}$ and $\hat{\delta}:= \max_{n\in \mathbb{N}}\delta_n$, in light of \eqref{Eq:SetFattEquiv} and the following observation.
\begin{lemma} \label{Lem:SetFattObserv}
	Let $(X,d)$ be a metric space and $A\subseteq X$ be a nonempty subset. 
	\begin{enumerate}[(a)]
		\item \label{Item:SetFatt1} For every $\delta,\eta \geq 0$, one has $(A^{(\delta)})^{(\eta)} \subseteq A^{(\delta+\eta)}$.
		\item \label{Item:SetFatt2} If $A\subseteq B$ and $\delta\geq 0$, then $A^{(\delta)}\subseteq B^{(\delta)}$.
		\item \label{Item:SetFatt3} Suppose in addition that $A$ is closed, and Let  $\{ A_n \}_{n=1}^\infty$ be a sequence of nonempty closed subsets of $X$ satisfying $ d_H(A_n,A)\to 0$ as $n\to \infty$. 
		Then, for every sequence $\{ \eta_n \}$ of nonnegative numbers with $\eta_n\to 0$ as $n\to \infty$, one has $\lim_{n\to \infty} d_H\big( A_n^{(\eta_n)},A \big)=0$.
		%\[  \lim_{n\to \infty} d_H\big( A_n^{(\eta_n)},A \big)=0  \quad \text{and} \quad  \lim_{n\to \infty}\mathsf{1}_{ A_n^{(\eta_n)} }(x)= \mathsf{1}_{A}(x) \; \; \text{for all} \; \; n\in \mathbb{N}. \] 

	\end{enumerate}
\end{lemma}

\begin{proof} $\\$
 (a)
 This follows directly from the triangle inequality and \eqref{Eq:SetFattDefn}. \\
(b) This is immediate from the definition \eqref{Eq:SetFattDefn}.\\
(c)
Denote $\delta_n:=d_H(A_n,A)$ for all $n\in \mathbb{N}$, and let $\eta_n$ be a sequence of nonnegative real numbers such that $\eta_n\to 0$ as $n\to \infty$. Since $A_n$ and $A$ are closed,
 \eqref{Eq:SetFattEquiv} implies that $A_n^{(\delta_n)}\supseteq A$ and $A^{(\delta_n)}\supseteq A_n$. It then follows from \eqref{Item:SetFatt1} that
\[ A_n^{(\delta_n+\eta_n)} \supseteq A_n^{(\delta_n)}\supseteq A \quad \text{and} \quad A^{(\delta_n+\eta_n)}\supseteq  \big( A^{(\delta_n)} \big)^{(\eta_n)}\supseteq A_n^{(\eta_n)}. \]
Applying \eqref{Eq:SetFattEquiv} again yields $d_H\big( A_n^{(\eta_n)},A \big)\leq \delta_n+\eta_n$. Since  $\delta_n+\eta_n\to 0$ as $n\to \infty$, the desired conclusion follows.
\end{proof}
\begin{remark}
	It should be noted the set inclusion in Lemma \ref{Lem:SetFattObserv}~\eqref{Item:SetFatt1} can be strict. As a simple example, let $X=\mathbb{N}_0$ equipped with the Euclidean distance. 
	For $A=\{ 0\}$ and $\delta,\eta=\frac{1}{2}$, we have
	\[  A^{(\frac{1}{2}+\frac{1}{2})}=\{ 0,1 \} \neq \{ 0 \}=\big( A^{(\frac{1}{2})} \big) ^{( \frac{1}{2} )}. \]
	Likewise, one can construct a metric space $(X,d)$, nonempty subsets $A\subsetneq B$ and $\delta>0$ such that  $A^{(\delta)}=B^{(\delta)}$, showing that the inclusion in Lemma \ref{Lem:SetFattObserv}~\eqref{Item:SetFatt2} need not be strict.
\end{remark}

Following the previous lemma, we have
\[  A_\infty^{ (2\hat{\delta}) }\supseteq \big( A_\infty^{(\hat{\delta})} \big)^{ (\hat{\delta}) } \supseteq  A_n^{(\hat{\delta})}\supseteq A_n^{ (\delta_n) }, \]
where $\delta_n:= d_H(A_n,A_\infty)$ and $\hat{\delta}:= \max_{n\in \mathbb{N}} \delta_n$. 
Hence, for any locally finite Borel measure $\mu$ on a proper metric space, the function $g:= \mathsf{1}_{ A_{\infty}^{(\hat{\delta})} }$ is $\mu$-integrable and $g(x)\geq \mathsf{1}_{ A_n^{(\delta_n)} }$ for every $x\in X$. 
We now show that $d_H(A_n,A_\infty)\to 0$ as $n\to \infty$ implies that $\lim_{n\to \infty} \mathsf{1}_{ A_n^{(\delta_n)} }(x)= \mathsf{1}_{A_\infty}(x)$ for all $x\in X$.
Let us first introduce the notions of set-theoretic limits for a sequence of subsets. For a sequence of subsets $\{ A_n \}_{n=1}^\infty$ we define its set-theoretic \emph{limit supremum} or \emph{limit infimum} by
\begin{equation} \label{Eq:SetLimits} 
	\limsup_{n\to \infty} A_n:=\big \{ x\in X: \; \limsup_{n\to \infty} \mathbf{1}_{A_n}(x)=1  \big \}	\quad \text{or} \quad \liminf_{n\to \infty} A_n:=\big \{ x\in X: \; \liminf_{n\to \infty} \mathbf{1}_{A_n}(x)=1  \big \}
\end{equation}
accordingly. Equivalently, 
\begin{equation*}
	\limsup_{n\to \infty} A_n=\bigcap_{n=1}^\infty \bigcup_{k=n}^\infty A_k
	\quad \text{and} \quad \liminf_{n\to \infty} A_n=\bigcup_{n=1}^\infty \bigcap_{k=n}^\infty A_k.  
\end{equation*}
When these two sets coincide, we denote their common value by
\begin{equation*}
	 \lim_{n\to \infty} A_n= \big\{ x\in X: \lim_{n\to \infty} \mathbf{1}_{A_n}(x)=1  \big\}.
\end{equation*}
Relying on the equivalent formulation of Hausdorff distance, using set fattenings in \eqref{Eq:SetFattEquiv}, we obtain the following lemma. 
\begin{lemma} \label{Lem:SetFattIncl}
	Let $(X,d)$ be a metric space and let $\{ A_n \}_{n=1}^\infty$ be a sequence of closed subsets such that $d_H(A_n,A_\infty)=0$ as $n\to \infty$, for some closed $A_\infty \subseteq X$. Then the  following assertions hold. 
	\begin{enumerate}[(a)]
		\item \label{Item:SetIncl1} One has $\limsup_{n\to \infty} A_n\subseteq A_\infty$ and $\limsup_{n\to \infty} \mathsf{1}_{A_n}(x)\leq \mathsf{1}_{A_\infty}(x)$ for all $x\in X$.
		\item \label{Item:SetIncl2} Denote $\delta_n:=d_H(A_n,A_\infty)$, and let $\{ \eta_n \}$ be a sequence of nonnegative numbers such that $\lim_{n\to \infty}\eta_n=0$ and $\eta_n\geq \delta_n$ for all $n\in \mathbb{N}$. Then $\lim_{n\to \infty}\mathsf{1}_{A_n^{(\eta_n)}}(x)= \mathsf{1}_{A_\infty}(x)$ for all $x\in X$.
	\end{enumerate} 
\end{lemma}
A version of Lemma~\ref{Lem:SetFattIncl}~\eqref{Item:SetIncl1} can be found in \cite{Beer74}, \cite{AvSim83} and \cite[Section~6]{Jit14}, where it is noted only in passing.
\begin{proof}$\\$
	(a) Let $x\in \limsup A_n$ and assume $x\notin A_\infty$. Since $A_\infty$ is closed, there exists some $\delta_0>0$ such that $d(x,A_\infty)>\delta_0>0$. Because $x\in \limsup A_n$, there exists a subsequence $\{ A_{n_k} \}_{k=1}^\infty$ with $x\in A_{n_k}$ for all $k$. Hence $d_H(A_{n_k}, A_\infty)\geq \frac{\delta_0}{2}$, contradicting the assumption that $d_H(A_n,A_\infty)\to 0$ as $n\to \infty$.
	Thus $\limsup_{n\to \infty} A_n \subseteq A_\infty$. The inequality for the indicator functions follows immediately from \eqref{Eq:SetLimits}.\\
	%\[  \limsup_{n\to \infty} A_n \subseteq A_\infty.   \]
	(b) By Lemma \ref{Lem:SetFattObserv}~\eqref{Item:SetFatt3}, we have $d_H\big( A_n^{(\eta_n)},A_\infty  \big)\to 0$ as $n\to \infty$. Applying part \eqref{Item:SetIncl1} to the sequence $\{ A_n^{(\eta_n)} \}_{n=1}^\infty$ yields  $\limsup_{n\to \infty} A_n^{(\eta_n)} \subseteq A_\infty$. Assume, for contradiction, that there exists an $x\in A_\infty$ such that $x\notin  \liminf_{n\to \infty} A_n^{(\eta_n)}$.     
	By \eqref{Eq:SetFattEquiv} and Lemma \ref{Lem:SetFattObserv}~\eqref{Item:SetFatt2}, one has $A_n^{(\eta_n)}\supseteq A_\infty$ for all $n\in \mathbb{N}$. Hence $x\in A_n^{(\eta_n)}$ for all $n\in \mathbb{N}$, so
	\[ x\in \bigcap_{n=1}^\infty A_n^{(\eta_n)}\subseteq \liminf_{n\to \infty} A_n^{(\eta_n)},   \]
	contradicting our assumption on $x$.
	Therefore $A_\infty\subseteq \liminf_{n\to \infty}A_n^{(\eta_n)}$, and since $\liminf_{n\to \infty}A_n^{(\eta_n)} \subseteq \limsup_{n\to \infty}A_n^{(\eta_n)}$, we conclude from \eqref{Eq:SetLimits} that 
	$ \lim_{n\to \infty} \mathbf{1}_{A_n^{(\eta_n)}}(x) = \mathbf{1}_{A_\infty}(x) $ for all $x\in X$.
	%$\lim_{n\to \infty} A_n^{(\eta_n)}=A_\infty$. The equality for the indicator functions again follows from \eqref{Eq:SetLimits}. 
	%\[  \lim_{n\to \infty} \mathsf{1}_{A_n^{(\eta_n)}}(x)=\mathsf{1}_{A_\infty}(x) \quad \text{for all} \quad x\in X. \]
\end{proof}
Before proving Theorem \ref{Thm:Main1}, we first state a  semicontinuity result for measures arising from convergence with respect to the Hausdorff distance.

\begin{proposition}
	Let $(X,d)$ be a metric space, let $\{ A_n \}_{n=1}^\infty$ be a sequence of closed subsets such that $d_H(A_n,A_\infty)\to 0$ as $n\to \infty$, and let $\mu$ be a Borel measure on $X$. If $\mu \big( \cup_{n=1}^\infty A_n \big)<\infty$, then $\mu(A_\infty)\geq \limsup_{n\to \infty} \mu(A_n)$. In particular, if there exists an $m>0$ such that $\mu(A_n)\geq m$ for infinitely many $n$, then $\mu(A_\infty)\geq m$.
\end{proposition}
	This semicontinuity result is standard. For instance,   a version of it appears in  \cite{Fil14}. 
\begin{proof}
	By Lemma \ref{Lem:SetFattIncl}~\eqref{Item:SetIncl1}, we have $\mathsf{1}_{A_\infty}(x) \geq \limsup_{n\to \infty} \mathsf{1}_{A_n(x)}$ for all $x\in X$.
	Applying the reverse Fatou's lemma to the  sequence  $\{ \mathsf{1}_{A_n} \}_{n=1}^\infty$, and noting  that $\mathsf{1}_{ \cup_{n=1}^\infty A_n }$ is $\mu$-integrable by assumption, we obtain
	\[  \mu(A_\infty)=\int_X \mathsf{1}_{A_\infty}d\mu \geq \int_X \limsup_{n\to \infty} \mathsf{1}_{A_n}d\mu \geq \limsup_{n\to \infty} \int_X \mathsf{1}_{A_n}d\mu = \limsup_{n\to \infty} \mu(A_n). \]
	If there exists an $m>0$ and an infinite subsequence $\{ A_{n_k} \}_{k=1}^\infty$ such that $\mu(A_{n_k})\geq m$ for all $k$,  then the preceding inequality implies $\mu(A_\infty)\geq \limsup_{n\to \infty} \mu(A_n) \geq m$, as claimed.
	%\[ \mu(A_\infty)\geq \limsup_{n\to \infty} \mu(A_n) \geq m, \]
	%as claimed.
\end{proof}

We finally turn to prove Theorem \ref{Thm:Main1}.
\begin{proof}[Proof of Theorem \ref{Thm:Main1}]
	As indicated earlier, the proof follows from the dominated convergence theorem. Let $\delta_n:= d_H(A_n,A_\infty)$ for all $n\in \mathbb{N}$, and set $\hat{\delta}:= \max_{n\in \mathbb{N}}\delta_n$. 
	Define $g:= \mathsf{1}_{A_\infty^{(2\hat{\delta})}}$, $f:=\mathsf{1}_{A_\infty}$ and $f_n:=\mathsf{1}_{A_n^{(\delta_n)}}$ for all $n\in \mathbb{N}$. By Lemma \ref{Lem:SetFattIncl}~\eqref{Item:SetIncl2}, we have $f_n(x)\to f(x)$ for all $x\in X$. 
	Moreover, by Lemma~\ref{Lem:SetFattObserv}~\eqref{Item:SetFatt1}, it follows that $f_n(x)\le g(x)$ for all $x\in X$.
	Since $X$ is a proper metric space and $\mu$ is finite on compact subsets, the function $g$ is $\mu$-integrable.
	Hence, by the dominated convergence theorem,
	\[ \lim_{n\to \infty} \mu(A_n^{(\delta_n)}) = \lim_{n\to \infty} \int_{X} \mathsf{1}_{A_n^{(\delta_n)}}d\mu  \overset{DCT}{=} \int_X \mathsf{1}_{A_\infty}d\mu  = \mu(A_\infty), \]
	as desired.
\end{proof}

\section{Estimates for approximations using periodic Schr\"{o}dinger operators} \label{Sect:Periodic}
In this section, we discuss an application of the previous results to for approximations of spectra by periodic Schr\"{o}dinger operators. A discrete Schr\"{o}dinger operator $H:\ell^2(\mathbb{Z}^d)\to \ell^2(\mathbb{Z}^d)$ is an operator of the form 
\begin{equation*}
	[H\psi](\mathsf{n})= \sum_{|| \mathsf{m}-\mathsf{n} ||=1 } \psi(\mathsf{m}) + V( \mathsf{n} ) \psi( \mathsf{n} ) \quad \text{for all} \quad \psi \in \ell^2(\mathbb{Z}^d) \quad \text{and} \quad \mathsf{n} \in \mathbb{Z}^d, 
\end{equation*}  
where $V:\mathbb{Z}^d\to \mathbb{R}$ is the potential function and $|| \cdot ||$ is the 1-norm on $\mathbb{R}^d$. We say that $H$ is a periodic operator if the potential function satisfies  
\begin{equation*}
	V(\mathsf{n}+p_j\mathsf{e}_j)= V(\mathsf{n}) \quad \text{for all} \quad \mathsf{n}\in \mathbb{Z}^d \; \; \text{and} \; \; 1\leq j \leq d,
\end{equation*}
for some positive integers $p_1,...,p_d\in \mathbb{N}$. We call the integers $p_1,...,p_d$ the periods of the operator $H$. Equivalently, the stabilizer of $V$,
\begin{equation} \label{Eq:StabDefn}
	\operatorname*{stab}(V):=\{ \mathsf{m}\in \mathbb{Z}^d: V(\mathsf{n}+\mathsf{m})= V(\mathsf{n}) \; \text{for all} \; \mathsf{n}\in \mathbb{Z}^d \},
\end{equation}
is a finite index subgroup of $\mathbb{Z}^d$.
Recall that the spectrum of a discrete self-adjoint periodic Schr\"{o}dinger operator $H:\ell^2(\mathbb{Z}^d)\to \ell^2(\mathbb{Z}^d)$ is of the form $ \operatorname*{spec}(H)= \cup_{i=1}^q [a_i,b_i] $. 
Consequently, for every $\delta>0$,
\begin{equation} \label{Eq:BandFatt}
	\big( \operatorname*{spec}(H) \big)^{(\delta)}= \bigcup_{i=1}^q [a_j-\delta,b_j+\delta].
\end{equation} 
We use this simple observation to prove Corollary \ref{Cor:Last1}.
\begin{proof}[Proof of Corollary \ref{Cor:Last1}]
	By \eqref{Eq:BandFatt}, we know that $\big( \operatorname*{spec}(H_n) \big)^{(\delta_n)}=\cup_{i=1}^{q_n}\big[a_i^{(n)}-\delta_n, b_i^{(n)}+\delta_n  \big]$. The desired statement follows from a direct application of Theorem \ref{Thm:Main1}.
\end{proof}
Corollary \ref{Cor:Last1} demonstrates that periodic Schrödinger approximations can be used to compute the measure of an approximated spectrum.
In a similar spirit, we obtain the following estimate.
\begin{Cor} \label{Cor:Last2}
	Let $A_\infty\subseteq \mathbb{R}$ be a nonempty compact set and let $\{ A_n \}_{n=1}^\infty \subseteq \mathcal{K}(\mathbb{R})$ be a sequence of nonempty compact  subsets, each consisting of $q_n$ connected components. If $\lim_{n\to \infty}\operatorname*{Leb}(A_n)$ exists and $\lim_{n\to \infty}q_nd_H(A_n,A_\infty)=0$, then $\operatorname*{Leb}(A_\infty)= \lim_{n\to \infty} \operatorname*{Leb}(A_n)$, where $\operatorname*{Leb}$ is the Lebesgue measure on $\mathbb{R}$.
\end{Cor}
Versions of this corollary have appeared previously in \cite{Last93,Last94,Jit02,Jit14}.
\begin{proof}
	Since $A_n\subseteq \mathbb{R}$ is compact and has $q_n$ many connected components, there exist sequences $\{ a_i^{(n)} \}_{i=1}^{q_n}, \{ b_i^{(n)} \}_{i=1}^{q_n}\subseteq \mathbb{R}$ such that $a_i^{(n)}\leq b_i^{(n)}$ for all $1\leq j \leq q_n$ and 
	\[ A_n= \bigcup_{i=1}^{q_n} [a_i^{(n)}, b_i^{(n)}]  . \]
	Denoting again $\delta_n:= d_H(A_n,A_\infty)$, we have $A_n^{(\delta_n)}=\bigcup_{i=1}^{q_n} [a_i^{(n)}-\delta_n, b_i^{(n)}+\delta_n]$. Since $\operatorname*{Leb}$ is subadditive, we conclude 
	\[ \operatorname*{Leb}(A_n^{(\delta_n)})\leq \sum_{i=1}^{q_n} \operatorname*{Leb}\big( [a_i^{(n)}-\delta_n, b_i^{(n)}+\delta_n]  \big)= \sum_{i=1}^{q_n} \big( b_i^{(n)}-a_i^{(n)}+2\delta_n \big)= \operatorname*{Leb}(A_n)+2q_n \delta_n. \]
	Combining the last estimate with Theorem \ref{Thm:Main1}, since we assumed $\lim_{n\to \infty}q_n\delta_n\to 0$ and $\lim_{n\to \infty} \operatorname*{Leb}(A_n)$ exists, we conclude that $\lim_{n\to \infty} \operatorname*{Leb}(A_n)= \operatorname*{Leb}(A_\infty)$.
\end{proof}

However, obtaining an explicit form for the spectrum of a  discrete periodic Schr\"{o}dinger operator is numerically challenging. 
We therefore propose a simpler numerical estimate than that provided by Corollary \ref{Cor:Last1}. To this end, we introduce the discrete Bloch--Floquet theory and recall several relevant facts.

Let $H:\ell^2(\mathbb{Z}^d)\to \ell^2(\mathbb{Z}^d)$ be a periodic discrete Schrödinger operator with periods $p_1,\ldots,p_d$.
The aforementioned periods  induce a decomposition of $\mathbb{Z}^d$ to disjoint translates of the fundamental cell $Q:=\prod_{j=1}^d\{ 0,....,p_j-1 \}$.
The volume of this fundamental cell is  $q:=\prod_{j=1}^{d}p_j$, and we denote $\mathsf{p}=(p_1,...,p_d)$ for later use. 
For each $\theta=(\theta_1,...,\theta_d) \in [0,1)^d$ %and $\mathsf{q}\in \mathbb{N}^d$
, we  consider the subspace $\ell^2_{(\mathsf{p},\theta)}(\mathbb{Z}^d)\subseteq \ell^2(\mathbb{Z}^d)$ consisting of functions satisfying the Floquet boundary condition
\begin{equation} \label{Eq:FloqCond}
	\psi(\mathsf{n}+p_j\mathsf{e}_j)= e^{2\pi i p_j \theta_j } \psi(\mathsf{n}) \quad \text{for all} \quad \mathsf{n}\in \mathbb{Z}^d \; \; \text{and} \; \; 1\leq j \leq d.
\end{equation}
Restricting $H$ to this subspace reduces the problem to a $q$-dimensional space $\{ \psi(\mathsf{n}): \mathsf{n} \in Q \}$, where finding generalized eigenfunctions satisfying \eqref{Eq:FloqCond} amounts to solving an eigenvalue problem for a $q\times q$ self-adjoint matrix $H(\theta)$. The family of matrices $\{ H(\theta) \}_{\theta\in \mathbb{T}^d}$ is called the Bloch--Floquet decomposition of $H$, and we write $H= \int_{\mathbb{T}^d}^{\oplus} H(\theta)$.
Using { Bloch--Floquet theory, see \cite[Theorem XIII.85]{SimonReedIV} }, the spectrum of the periodic operator $H$ satisfies
\begin{equation*}
	\operatorname*{spec}(H)= \bigcup_{\theta \in [0,1)^d} \operatorname*{spec}(H(\theta)).
\end{equation*}
Moreover, if $	\lambda_1 (\theta) \leq \lambda_2(\theta) \leq ... \leq \lambda_{q}(\theta)$ denote the ordered eigenvalues of $H(\theta)$, then each $\lambda_i(\theta)$
depends continuously on $\theta$, and 
\begin{equation*}
	\operatorname*{spec}(H)= \bigcup_{i=1}^q \Big[ \min_{\theta \in [0,1)^d}  \lambda_i(\theta),  \max_{\theta \in [0,1)^d}  \lambda_i(\theta) \Big].
\end{equation*}
Consult \cite[Theorem XIII.99]{SimonReedIV} or \cite[Corollary 2.4]{Fil21} for further details.

\begin{proposition}{\cite{Krug11}} \label{Prop:Krug}
		Let $H$ be a discrete periodic Schr\"{o}dinger operator on $\ell^2(\mathbb{Z}^d)$, which is $p_j \mathsf{e}_j$-periodic for all $1\leq j\leq d$.
		Set $q:= \prod_{j=1}^d p_j$. 
		Then, for the Bloch--Floquet decomposition $H= \int_{\mathbb{T}^d}^\oplus H(\theta)$, the eigenvalues satisfy
		\begin{equation} \label{Eq:KrugBound}
			\max_{\theta \in [0,1)^d}  \lambda_i(\theta)- \min_{\theta \in [0,1)^d}  \lambda_i(\theta)\leq \sum_{j=1}^{d} \frac{4\pi}{p_j} \quad \text{for all} \quad 1\leq i \leq q.
		\end{equation}
\end{proposition}
   Using the bandwidth estimate in Proposition \ref{Prop:Krug}, the following lemma provides a spectral covering of controlled diameter for periodic approximations.
   \begin{lemma} \label{Lem:PerCover}
   	Let $H:\ell^2(\mathbb{Z}^d)\to \ell^2(\mathbb{Z}^d)$  
   	be a bounded self-adjoint discrete Schrödinger operator, and let $H_{\operatorname*{per}}:\ell^2(\mathbb{Z}^d)\to \ell^2(\mathbb{Z}^d)$ be a  periodic self-adjoint discrete Schr\"{o}dinger operator with periods $p_1 \mathsf{e}_1,..., p_d \mathsf{e}_d \in \mathbb{Z}^d$.
   	Define $\delta:=d_H\big( \operatorname*{spec}(H_{\operatorname*{per}}), \operatorname*{spec}(H) \big)$ and $r:= \sum_{j=1}^d \frac{4\pi}{p_j}$. Then $\big( \operatorname*{spec}(H_{\operatorname*{per}} ) \big)^{(\delta)}$ induces a $(2\delta+r)$-cover of $\operatorname*{spec}(H)$. 
   	Moreover, if the Bloch--Floquet decomposition of $H_{\operatorname*{per}}$ is
   	\[ H_{ \operatorname*{per} }= \int_{\mathbb{T}^d}^\oplus H_{ \operatorname*{per} }(\theta) d\theta,  \] 
   	then the set $\big( \operatorname*{spec}(H_{ \operatorname*{per} }(\theta)) \big)^{(r+\delta)}$ induces a $2(\delta+r)$-cover of $\operatorname*{spec}(H)$ for all $\theta\in [0,1)^d$.
   \end{lemma}
   \begin{proof}
   	We recall that the spectra of $H_{\operatorname*{per}}(\theta)$ are given by $$\operatorname*{spec}(H_{ \operatorname*{per} })=\bigcup_{i=1}^{q}  [\check{\lambda}_i, \hat{\lambda}_i] ,$$
   	where $\hat{\lambda}_i:=\max_{\theta \in [0,1)^d} \lambda_i(\theta)$ and $\check{\lambda}_i:=\min_{\theta \in [0,1)^d} \lambda_i(\theta)$ for all $1\leq i \leq q$.
   	Proposition \ref{Prop:Krug} implies that $\max_{1\leq i \leq q} \hat{\lambda}_i-\check{\lambda}_i \leq r$. The collection $\big\{  [ \check{\lambda}_i-\delta, \hat{\lambda}_i + \delta ]  \big\}_{i=1}^q$ forms a $(2\delta+r)$-cover of $\operatorname*{spec}(H)$ by \eqref{Eq:SetFattEquiv}. %Applying Proposition \ref{Prop:Krug}, we see that $r_n\leq \sum_{j=1}^{d} \frac{4\pi}{q^{(j)}_n}$. Since $H$ {\color{red} has no period} and $H_n$ converges strongly to $H$, it must follow that $q_n^{(j)}\to \infty$ as $n\to \infty$, for all $1\leq j \leq d$. This implies in particular that $r_n\to 0$ as $n\to \infty$.
   	\\
   	Applying Proposition \ref{Prop:Krug} again, we see that $  \big( \operatorname*{spec}(H_{\operatorname*{per}}) \big)^{(\delta)}\subseteq \big( \operatorname*{spec}(H_{\operatorname*{per}}(\theta)) \big)^{(r+\delta)}$ for all $\theta \in [0,1)^d$. Moreover, since $\operatorname*{spec}(H) \subseteq \big( \operatorname*{spec}(H_{\operatorname*{per}}) \big)^{(\delta)}$ the collection $\big\{  \big[ \lambda_i(\theta) -(\delta+r), \lambda_i(\theta) + (\delta+r)  \big]  \big\}$ 
   forms a $2(\delta+r)$-cover of $\operatorname*{spec}(H)$ by \eqref{Eq:SetFattEquiv}.
   \end{proof}
   
   The last lemma implies the existence of spectral covers for the spectrum of a periodically approximated operator. Such spectral covers, provided by Lemma~\ref{Lem:PerCover}, are essential for obtaining computationally feasible estimates of the spectra of strongly aperiodic operators.
   A discrete Schrödinger operator $H:\ell^2(\mathbb{Z}^d)\to \ell^2(\mathbb{Z}^d)$ with potential function $V$ is called \emph{strongly aperiodic} if, for every nonzero $\mathsf{m}\in \mathbb{Z}^d$, there exists an $\mathsf{n}\in \mathbb{Z}^d$ such that $V(\mathsf{n}+\mathsf{m})\neq V(\mathsf{n})$. 
   Equivalently, the stabilizer of $V$ is the trivial subgroup $\{ \mathsf{0} \}$. 
   Using Theorem \ref{Thm:Main1} together with Lemma \ref{Lem:PerCover}, we obtain the following result for strongly aperiodic discrete operators.
   \begin{proposition} \label{Prop:NumerApprox}
   	Let $H:\ell^2(\mathbb{Z}^d)\to \ell^2(\mathbb{Z}^d)$ be a bounded self-adjoint strongly aperiodic Schrödinger operator, and let $\mu$ be a locally finite Borel measure on $\mathbb{R}$. Assume there exists a sequence of periodic self-adjoint discrete Schrödinger operators $H_n:\ell^2(\mathbb{Z}^d)\to \ell^2(\mathbb{Z}^d)$ such that $H_n$ converges in the strong operator topology to $H$ and $\lim_{n\to \infty} d_H\big( \operatorname*{spec}(H_n) , \operatorname*{spec}(H) \big)=0$.  
   	Then there exists a sequence $\{ r_n \}_{n=1}^\infty$ with $\lim_{n\to \infty}r_n=0$ such that, for any $\theta\in [0,1)^d$,
   	\begin{equation*}
   		\mu( \operatorname*{spec}(H) )= \lim_{n\to \infty} \mu \Big( \bigcup_{i=1}^{q_n} \big[ \lambda_i^{(n)}(\theta)-(\delta_n +r_n),\lambda_i^{(n)}(\theta)+(\delta_n +r_n)   \big]   \Big)
   	\end{equation*}
   	where $\lambda_1^{(n)}(\theta)\leq ... \leq \lambda_{q_n}^{(n)}(\theta)$ are the eigenvalues of the fiber matrix $H_n(\theta)$ and $\delta_n:= d_H\big(  \operatorname*{spec}(H_n), \operatorname*{spec}(H) \big)$.
   \end{proposition}
   This proposition provides a numerically oriented version of Corollary~\ref{Cor:Last1}.
   It is particularly suitable for computational studies, since the direct evaluation of spectra for multidimensional periodic Schrödinger operators via Bloch--Floquet theory is often computationally demanding. See also \cite[Example~3.1]{CEF24} for comparison. The assumptions of Proposition~\ref{Prop:NumerApprox} arise naturally in the context of dynamically defined operators, where the underlying dynamical system is itself approximated. 
   For an overview of one-dimensional dynamically defined operators, see \cite{DF22,DaFi24-book_2}. Spectral approximations via dynamical approximations are treated systematically in \cite{BBdN18}, with quantitative refinements given in \cite{BBC19,BecTak21}. Conditions ensuring successful dynamical approximation in the substitution setting appear in \cite{BBdN20,BBPT}.
   \begin{proof} [Proof of Proposition \ref{Prop:NumerApprox}]
   	For every $n\in \mathbb{N}$, there exist $p_1^{(n)}, ...,p_d^{(n)}$ such that $H_n$ is periodic with respect to $p_j \mathsf{e}_j$ for all $1\leq j \leq d$. It follows that $q_n= \prod_{j=1}^{d}p_j^{(n)}$ for all $n\in \mathbb{N}$. Denote $r_n:= \sum_{j=1}^{d}\frac{4\pi}{p_j^{(n)}}$ for all $n\in \mathbb{N}$. It follows from Lemma \ref{Lem:PerCover} that $\operatorname*{spec}(H)\subseteq \bigcup_{i=1}^{q_n}\big[ \lambda_i^{(n)}(\theta)- (\delta_n+r_n), \lambda_i^{(n)}(\theta)+(\delta_n+r_n)  \big]$. Since $H$  is strongly aperiodic, it must follow that $p_j^{(n)}\to \infty$ as $n\to \infty$, for all $1\leq j\leq d$. In particular, we conclude that $r_n\to 0$ as $n\to \infty$. Moreover, it follows that $d_H( \operatorname*{spec}\big( H_n(\theta), \operatorname*{spec}(H) \big) )\leq  2(r_n+q_n)$ for all $n\in \mathbb{N}$ and $r_n+q_n \overset{n\to \infty}{\to}0$. By \eqref{Eq:BandFatt} and Theorem \ref{Thm:Main1}, we obtain the desired equality.
   \end{proof} 
   
   Using Lemma \ref{Lem:PerCover}, we can also restate a result of Last \cite{Last93,Last94} concerning an upper bound on the fractal dimension.
   For  $\alpha\geq 0$ and $\delta>0$, we define the \emph{$(\alpha,\delta)$-Hausdorff content} of a Borel set $E\subseteq \mathbb{R}^d$ as
   \begin{equation*}
   	\mathcal{H}_\delta^{\alpha}(E)= \inf \Big\{ \sum_{m=1}^\infty \operatorname*{diam}(C_m)^\alpha : \exists \{ C_m \}_{m=1}^\infty, \;  E \subseteq \cup C_m, \; \operatorname*{diam}(C_m)\leq \delta  \Big\}.
   \end{equation*}
   Letting $\delta\to 0 ^+$ yields a monotone limit $\mathcal{H}^\alpha(E):= \lim_{\delta\to 0^+}\mathcal{H}_\delta^{\alpha}(E)$, which is called the \emph{$\alpha$-Hausdorff measure} of $E$. See \cite[Chapter 1.2]{Falcon90} for more details. A classic result states that for every set $E\subseteq \mathbb{R}^d$, there exists a critical value $\alpha_E\in [0,d]$ such that $\mathcal{H}^{\alpha}(E)=0$ for $\alpha< \alpha_E$ and  $\mathcal{H}^{\alpha}(E)=\infty$ for $\alpha> \alpha_E$. This number $\alpha_E$ is called the \emph{Hausdorff dimension} of $E$. For more details on the Hausdorff measure and dimension, consult \cite{Falcon90} and the references therein.\\ 
   The following observation follows directly from these definitions. Suppose $\{ \mathcal{C}_n \}_{n=1}^\infty$ 
   is a sequence of finite covers of $E$, where $\mathcal{C}_n=\{ I_{j,n} \}_{i=1}^{q_n}$ is an  $\eta_n$-cover of $E$, then 
   \begin{equation} \label{Eq:HausDimObserv}
   	\eta_n  \overset{n\to \infty}{\to} 0^+ \; \; \text{and} \;\; \limsup_{n\to \infty} \sum_{i=1}^{q_n} \big(\operatorname*{Leb}(I_{j,n}) \big)^\alpha <\infty
   	\quad \Longrightarrow \quad
   	\mathcal{H}^\alpha(E)<\infty \;\; \text{and} \;\; \dim_H(E)\leq \alpha.
   \end{equation}
   The next result provides an upper bound on the Hausdorff dimension of a limit set approximated by finite unions of intervals.
   
   \begin{lemma} \label{Lem:DimUpBd}
   	Let $\{ A_n \}_{n=1}^\infty \subseteq \mathcal{K}(\mathbb{R})$ be a sequence of sets of the form $A_n=\cup_{i=1}^{q_n}[a_i^{(n)},b_i^{(n)}]$ converging in the Hausdorff distance to a compact set $A\in \mathcal{K}(\mathbb{R})$. Denote $\delta_n:= d_H(A_n,A)$ for all $n\in \mathbb{N}$.
   	Then the following assertions hold.
   	\begin{enumerate}[(a)]
   		\item \label{Item:DimLast} If there exists a  constant $\beta>0$ such that $q_n\overset{n\to \infty}{\to}\infty$ and $\limsup_{n\to \infty} \operatorname*{Leb}\big( A_n^{(\delta_n)} \big)\cdot q_n^{\beta}<\infty$, then $\mathcal{H}^{\frac{1}{1+\beta}}(A)<\infty$ and $\dim_H(A)\leq \frac{1}{1+\beta}$.
   		\item \label{Item:DimDirect} Let $r_n:= \max_{1\leq  j \leq n}(b_i^{(n)}-a_i^{(n)})$ for all $n\in \mathbb{N}$. If there exists a constant $\alpha>0$ such that  $\limsup_{n\to \infty} (2\delta_n+r_n)^\alpha q_n< \infty$ and $r_n \overset{n\to \infty}{\to}0$, then $\mathcal{H}^{\alpha}(A)<\infty$ and $\dim_H(A)\leq \alpha$.
   	\end{enumerate}

   \end{lemma}
   A version of Lemma~\ref{Lem:DimUpBd},\eqref{Item:DimLast} appears in~\cite[Lemma~5.1]{Last94}.
   Moreover, a refined estimate based on~\cite{Last94} shows that the Hausdorff dimension of the spectrum of the almost Mathieu operator is zero for a $G_\delta$ set of frequencies. See~\cite{Last16} for details.
   \begin{proof}$~$\\
   	(a) First note that $A_n^{(\delta_n)}=\cup_{i=1}^{\tilde{q}_n} I_{j,n}$, where $\tilde{q}_n$ is the number of connected components of $A_n^{(\delta_n)}$ and each $I_{j,n}$ is a closed interval. Using H\"{o}lder's inequality, we can see  that
   	\begin{equation*}
   		\sum_{i=1}^{q_n} \operatorname*{Leb}\big( I_{j,n} \big)^{\frac{1}{t}}\leq \Bigg( \sum_{i=1}^{q_n} \operatorname*{Leb}\big( I_{j,n}\big) \Bigg)^{\frac{1}{t}}  \Bigg( \sum_{i=1}^{\tilde{q}_n}1 \Bigg)^{1-\frac{1}{t}}= \Big( \operatorname*{Leb}\big( A_n^{(\delta_n)} \big) \Big)^{\frac{1}{t}} \tilde{q}_n^{1-\frac{1}{t}}, 
   	\end{equation*}
   	for all $t\geq 1$. Since $\tilde{q}_n\leq q_n$ and $t \geq 1$, it also follows that
   	\begin{equation*}
   		\sum_{i=1}^{q_n} \operatorname*{Leb}\big( I_{j,n} \big)^{\frac{1}{t}}\leq \Big( \operatorname*{Leb}\big( A_n^{(\delta_n)} \big) \Big)^{\frac{1}{t}} q_n^{1-\frac{1}{t}},
   	\end{equation*}
   	while $\{ I_{j,n} \}_{i=1}^{ \tilde{q}_n }$ is a cover of $A_n^{(\delta_n)}$. By \eqref{Eq:SetFattEquiv}, it is also an cover of $A$.
   	Using our assumption on $\beta>0$, there exists a constant $\hat{C}>0$ such that $\operatorname*{Leb}\big( A_n^{(\delta_n)} \big)\leq \frac{\hat{C}}{q_n^\beta}$ for all $n\in \mathbb{N}$. Hence, 
   	\begin{equation*}
   		\sum_{i=1}^{q_n} \operatorname*{Leb}\big( I_{j,n} \big)^{\alpha}\leq \hat{C}^{ \alpha } q_n^{ 1 -\alpha(1+\beta) }, \quad \text{for} \; \; \alpha =\frac{1}{t}.
   	\end{equation*}
   	Observing that each $\{ I_{j,n} \}_{i=1}^{ \tilde{q}_n }$ is an $\eta_n$-cover, for $\eta_n \leq \operatorname*{Leb}\big( A_n^{(\delta_n)} \big)$, we conclude from  \eqref{Eq:HausDimObserv} that $\mathcal{H}^\alpha(A)<\infty$ for all $\alpha \leq \frac{1}{1+\beta}$ and $\dim_H(A)\leq \frac{1}{1+\beta}$.
   	\\
   	(b) By \eqref{Eq:SetFattEquiv}, we can see that $\mathcal{C}_n= \big\{ [ a_i^{(n)}-\delta_n, b_i^{(n)}+\delta_n]  \big\}$ 
   	is a $(2\delta_n +r_n)$-cover of $A$. By our assumptions, we have $(2\delta_n+r_n)\overset{n\to \infty}{\to}0$ and $q_n(2\delta_n+r_n)^\alpha\geq \mathcal{H}_{(2\delta_n+r_n)}^{(\alpha)}(A)$. The desired result now follows from a direct application of \eqref{Eq:HausDimObserv}.
   \end{proof}
   %\begin{remark}
   	%Note that Lemma \ref{Lem:DimUpBd}~\eqref{Item:DimDirect} follows from Lemma \ref{Lem:DimUpBd}~\eqref{Item:DimLast} by setting $\alpha=\frac{1}{1+\beta}$.
   	%However, since Lemma \ref{Lem:DimUpBd}~\eqref{Item:DimDirect} also admits a more direct proof, we have chosen to present an independent argument.
   %\end{remark}

	\bibliographystyle{alpha}
	\bibliography{HausMeasSub1}
	
\end{document}